\documentclass[11pt, reqno]{amsart}
\pdfoutput=1
\usepackage[hmargin=2.5cm,vmargin=3.5cm]{geometry}

\usepackage{amsmath,amssymb}
\usepackage{mathrsfs}
\usepackage{amsrefs}
\usepackage{graphicx,color}
\usepackage{amsthm}
\usepackage{latexsym}
\usepackage[all]{xy}
\usepackage{hyperref}
\usepackage[mathscr]{eucal}
\usepackage{comment}

\usepackage{tcolorbox}

\usepackage{tikz-cd}
\usepackage{fancyhdr}

\makeatletter
\@addtoreset{equation}{section}
\makeatother

\BibSpec{book}{%
    +{}  {\PrintPrimary}                {transition}
    +{.} { \textit}                     {title}
    +{.} { }                            {part}
    +{:} { \textit}                     {subtitle}
    +{,} { \PrintEdition}               {edition}
    +{}  { \PrintEditorsB}              {editor}
    +{,} { \PrintTranslatorsC}          {translator}
    +{,} { \PrintContributions}         {contribution}
    +{,} { }                            {series}
    +{,} { \voltext}                    {volume}
    +{,} { }                            {publisher}
    +{,} { }                            {organization}
    +{,} { }                            {address}
    +{,} { }                            {status}
    +{,} { \PrintDOI}                   {doi}
    +{,} { \PrintISBNs}                 {isbn}
    +{}  { \parenthesize}               {language}
    +{}  { \PrintTranslation}           {translation}
    +{;} { \PrintReprint}               {reprint}
    +{,} { \PrintDate}                  {date}
    +{.} { }                            {note}
    +{.} {}                             {transition}
}
\BibSpec{article}{%
    +{}  {\PrintAuthors}                {author}
    +{,} { \textit}                     {title}
    +{.} { }                            {part}
    +{:} { \textit}                     {subtitle}
    +{,} { \PrintContributions}         {contribution}
    +{.} { \PrintPartials}              {partial}
    +{,} { }                            {journal}
    +{}  { \textbf}                     {volume}
    +{}  { \PrintDatePV}                {date}
    +{,} { \issuetext}                  {number}
    +{,} { \eprintpages}                {pages}
    +{,} { }                            {status}
    +{,} { \PrintDOI}                   {doi}
    +{}  { \parenthesize}               {language}
    +{}  { \PrintTranslation}           {translation}
    +{;} { \PrintReprint}               {reprint}
    +{.} { }                            {note}
    +{,} { \eprint}                     {eprint}
    +{.} {}                             {transition}
}
\BibSpec{collection.article}{%
    +{}  {\PrintAuthors}                {author}
    +{,} { \textit}                     {title}
    +{.} { }                            {part}
    +{:} { \textit}                     {subtitle}
    +{,} { \PrintContributions}         {contribution}
    +{,} { \PrintConference}            {conference}
    +{}  {\PrintBook}                   {book}
    +{,} { }                            {booktitle}
    +{,} { \PrintDateB}                 {date}
    +{,} { pp.~}                        {pages}
    +{,} { }                            {publisher}
    +{,} { }                            {organization}
    +{,} { }                            {address}
    +{,} { }                            {status}
    +{,} { \PrintDOI}                   {doi}
    +{,} { \eprint}        {eprint}
    +{}  { \parenthesize}               {language}
    +{}  { \PrintTranslation}           {translation}
    +{;} { \PrintReprint}               {reprint}
    +{.} { }                            {note}
    +{.} {}                             {transition}
}

\BibSpec{misc}{
  +{}{\PrintAuthors}  {author}
  +{,}{ \textit}      {title}
  +{,}{ }             {note}
  +{,}{ }             {date}
  +{.}{}              {transition}
}

\theoremstyle{definition}
\newtheorem{dfn}[equation]{Definition}
\newtheorem{exm}[equation]{Example}
\theoremstyle{plain}
\newtheorem{thm}[equation]{Theorem}
\newtheorem{prop}[equation]{Proposition}
\newtheorem{lem}[equation]{Lemma}
\newtheorem{cor}[equation]{Corollary}

\theoremstyle{remark}
\newtheorem{rmk}[equation]{Remark}

\newcommand{\cF}{\mathcal F}
\newcommand{\cG}{\mathcal G}

\newcommand{\cS}{\mathcal S}
\newcommand{\cT}{\mathcal T}

\newcommand{\cZ}{\mathcal Z}

\newcommand{\bC}{\mathbb{C}}
\newcommand{\bR}{\mathbb{R}}
\newcommand{\bN}{\mathbb{N}}

\newcommand{\bZ}{\mathbb{Z}}

\DeclareMathOperator{\Sp}{Sp}

\newcommand{\acts}{\curvearrowright} 

\newcommand{\crit}[1]{\pi_{#1,\mathrm{mc}}} 
\DeclareMathOperator{\critical}{crit}   
\newcommand{\pmc}{\pi_{\mathrm{mc}}}  
\newcommand{\harm}{\chi} 
\newcommand{\gauge}{\alpha} 

\title[Limits of KMS states on $\cT C^*(E)$]{Limits of KMS states on Toeplitz algebras of finite graphs}

 \author[T. Takeishi]{Takuya Takeishi}
 \address[Takuya Takeishi]{
  Faculty of Arts and Sciences\\
  Kyoto Institute of Technology\\
  Matsugasaki, Sakyo-ku, Kyoto\\
  Japan}
 \email[T. Takeishi]{takeishi@kit.ac.jp}

\begin{document}

\begin{abstract}
The structure of KMS states of Toeplitz algebras associated to finite graphs equipped with the gauge action is determined by an Huef--Laca--Raeburn--Sims. Their results imply that extremal KMS states of type I correspond to vertices, while extremal KMS states at critical inverse temperatures correspond to minimal strongly connected components. The purpose of this article is to clarify the role of non-minimal components and the relation between vertices and minimal components in the KMS-structure. For each component $C_0$ and each vertex $v \in C_0$, the KMS states at the critical inverse temperature of $v$ obtained by the limit of type I KMS states associated to $v$ uniquely decomposes into a convex combination of KMS states associated to minimal components. We show that for each minimal component $C$, the coefficient of the KMS state associated to $C$ is nonzero if and only if there exists a maximal path from $C$ to $C_0$ in the graph of components.   
\end{abstract}

\maketitle

\section{Introduction}
The research of equilibrium states of graph algebras originates from the work of Olesen--Pedersen \cite{OlePed}, in which they found the unique existence of KMS states on Cuntz algebras equipped with the gauge action. Their work has been generalised to Cuntz--Krieger algebras associated to irreducible matrices by Enomoto--Fujii--Watatani \cite{EFW}. Their method is an application of Perron--Frobenius theory, which is an essence even for the case of general finite graphs. As a consequence, it is known that if $E$ is a finite strongly connected graph, then there exists a unique KMS state on $C^*(E)$ with respect to the gauge action
at $\beta = \log \rho$, where $\rho$ is the spectral radius of the adjacency matrix of $E$. 
However, the KMS-structure is more complicated when a finite graph is reducible. KMS states of arbitrary finite graph algebras have been completely determined in the work of an Huef--Laca--Raeburn--Sims \cite{aHLRS15} combined with the recent work by the author and C. Bruce \cite[Section 5]{BT1}. As a consequence, if $E$ is a finite graph, then there is a one-to-one correspondence between extremal KMS states of type III and minimal strongly connected components of $E$, see Theorem \ref{thm:aHLRS15}. Here, the curious point is that not all components contribute to type III KMS states. 

On the other hand, we can observe a phase transition phenomenon by extending the C*-dynamical system to the Toeplitz algebra $\cT C^*(E)$ of a finite graph $E$. In the C*-dynamical system $\cT C^*(E)$ with the gauge action, each vertex corresponds to a chain of extremal type I KMS states $\varphi_{\beta,v}$ for inverse temperatures $\beta$ such that the partition function $Z_v(\beta)$ converges, see Section \ref{ssec:typeI} for the detail. Hence, the shape of KMS-simplex of $\cT C^*(E)$ spectaculary changes at each critical inverse temperature $\beta_v$. 

Comparing the structure of type I and type III KMS states, we come up with two natural questons: 
\begin{enumerate}
\item What is the role of non-minimal components in the KMS-structure?
\item What is the relation between type I and type III KMS states? 
\end{enumerate}
For the second question, we can expect a natural relation between type I KMS states of vertices $v$ and type III KMS states of components containing $v$, if $v$ belongs to a minimal component. However, the relation is not obvious when $v$ belongs to a non-minimal component, because non-minimal components do not give rise to KMS states. 

The aim of this article is to give answers to these questions. For the second question, it is natural to consider the limit of $\varphi_{\beta,v}$ as $\beta \to \beta_v+0$ with respect to the weak*-toplogy in the state space. Then we obtain a KMS state $\varphi_{v}$ at the critical inverse temperature for each vertex $v$. Here, the point is that $\varphi_{v}$ is not necessarily extremal. Since there are only finitely many extremal points in the KMS-simplices of $\cT C^*(E)$ for a finite graph $E$, the KMS state $\varphi_v$ may decompose into a convex combination of extremal KMS states $\psi_C$ associated to minimal components $C$. We first observe that $\varphi_{v}$ is a convex combination of $\psi_C$ for minimal components $C$ admitting a path from $C$ to $v$ in $E$ and with spectral radius $e^{\beta_v}$  (Corollary \ref{cor:coefficients}). Then, the main theorem of this article is the determination of minimal components $C$ such that the coefficient of $\psi_C$ is positive in the convex combination of $\varphi_v$ (Theorem \ref{thm:main}). 
In order to formulate the main theorem, we introduce a graph of strongly connected components $B(E,C_0)$ associated to a finite graph $E$ and a fixed strongly connected component $C_0$ of $E$. As the main theorem, we show that the coefficient of $\psi_C$ is positive if and only if there is a ``maximal path" (Definition \ref{def:maximal} \eqref{item:path}) from $C$ to $C_0$ in $B(E,C_0)$, where $C_0$ is the strongly connected component containing $v$. The information of ``intermediate" critical components (Definition \ref{def:maximal} \eqref{item:critical}) is in fact contained in the maximality paths, and is reflected to the convex decomposition of $\varphi_v$. Hence, this result gives a complete answer to the second question, and a partial answer to the first question. The role of non-minimal non-critical components are still not clear, which may be an interesting problem. 

In addition, we verify the effect of intermediate critical components in a concrete example (Example \ref{example:intermediate}). In the main theorem, we have not paid attention to the values of coefficients when they are positive. However, we may expect that the values of coefficients in the convex decomposition have more information. 
We examine this possibility in a concrete example (Example \ref{example:further}), which tells us that minimal components with spectral radius less than $e^{\beta_v}$ have an effect to the coefficients in the convex decomposition of $\varphi_v$. 
For a future direction, it is an  interesting problem to catch a qualitative feature of the convex decomposition which validates the phenomenon in Example \ref{example:further}. 

\subsection*{Acknowledgements}
This work is supported by JSPS KAKENHI Grant Number JP19K14551 (2019-2023). 

\section{Preliminaries}

In this section, we collect known results of KMS states of Toeplitz algebras of graphs. 
We fix basic notations at first. 
The symbol $\bR_+$ denotes the set of positive real numbers. 
For a square matrix $A$ of finite size, its spectral radius is denoted by $\rho(A)$. 
For a set $S$, its cardinality is denoted by $|S|$. 

For the definition of Toeplitz algebra of graphs, there are two different conventions, such as \cite{CarLar} and \cite{aHLRS13}. In this article, we adopt the convention in \cite{CarLar}. Note that several references in this article adopt the opposite convention. 

\subsection{Directed graphs, Toeplitz algebras, and KMS states}
A \emph{directed graph} $E=(E^0,E^1,s,r)$ consists of a set of vertices $E^0$, a set of edges $E^1$, source and range maps 
$s,r \colon E^1 \to E^0$. A \emph{finite path} $\mu$ in the graph $E$ is a finite sequence $\mu=e_0e_1 \cdots e_{n-1}$ of edges such that $r(e_i)=s(e_{i+1})$ for $i=0,1,\cdots n-1$. For a path $\mu=e_0e_1 \cdots e_{n-1}$, $|\mu| := n$ is called the \emph{length} of $\mu$, and the set of paths of length $n$ is denoted by $E^n$. Here, a vertex is considered as a path of length $0$. Let $E^* = \bigcup_{n=0}^\infty E^n$ be the set of all finite paths in $E$. If $\mu,\nu \in E^*$ satisfies $r(\mu)=s(\nu)$, then $\mu$ and $\nu$ are composable in the obvious way, and the composition is denoted by $\mu \nu$. An \emph{infinite path} $\mu$ in $E$ is an infinite sequence $\mu=e_0e_1 \cdots$ of edges such that $r(e_i)=s(e_{i+1})$ for $i=0,1,\cdots$. The set of all infinite paths is denoted by $E^\infty$, and let 
$E^{\leq \infty} := E^* \cup E^\infty$ be the set of all finite or infinite paths. We extend the definition of the composition $\mu \nu$ 
for $\mu \in E^*$ and $\nu \in E^{\leq \infty}$ in the natural way. The space $E^{\leq \infty}$ is a compact space with respect to the pointwise convergence topology (see \cite[Section 2]{Web} for the detail). 
A graph $E$ is called \emph{finite} if both $E^0$ and $E^1$ are finite. In this article, we restrict our attention to finite graphs. 

We fix notations related to graphs. Let $E$ be a finite graph. 
For $v,w \in E^0$, let $vE^*, E^*w, vE^*w$ denote the set of finite paths starting from $v$, the set of finite paths ending at $w$, the set of finite paths from $v$ to $w$ respectively. For a natural number $n$, $vE^n, E^nw, vE^nw$ denote the sets of paths of length $n$ defined in the same way. 
A \emph{strongly connected component} $C$ of $E$ is an equivalence class of $E^0$ with respect to the equivalence relation $\sim$ defined by $v \sim w$ if $vE^*w \neq \emptyset$ and $wE^*v \neq \emptyset$. 
The set of all strongly connected components of $E$ is denoted by $\pi(E)$. 
Since we do not consider other sorts of connected components in this article, strongly connected components are simply called \emph{components}. The graph $E$ is called \emph{strongly connected} if $|\pi(E)|=1$. 
For a subset $S \subset E^0$, let $E_S$ be the \emph{restriction} of $E$ to $S$. That is, $E_S$ is a graph with the vertex set $E_S^0 = S$, and the edge set $E_S^1$ is equal to the set of edges $e \in E^1$ such that $s(e), r(e) \in S$. 
The matrix $A_E = [|vE^1w|]_{v,w \in E^0}$ indexed by $E^0$ is called the \emph{adjacency matrix} of $E$. For subsets $S,T \subset E^0$, the adjacency matrix of $E_S$ is denoted by $A_S$, and the $S\times T$ block of $A_E$ is denoted by $A_{S,T}$. Note that $A_S = A_{S \times S}$. For a subset $D \subset E^0$ define $S_D := s(E^*D)$. 

Let $E$ be a directed graph. The \emph{Toeplitz algebra of $E$}, which is first defined in \cite[Section 4]{FowRae}, is the universal C*-algebra $\cT C^*(E)$ generated by projections $p_v$ for $v \in E^0$ and partial isometries $s_e$ for $e \in E^1$ subjecting to the relations 
\[ \begin{cases}
s_e^*s_e = p_{r(e)} & \mbox{ for all } e \in E^1, \\
(s_es_e^*)(s_fs_f^*) = 0 & \mbox{ for all } e, f \in E^1 \mbox{ with } e \neq f, \\
p_vp_w = 0 & \mbox{ for all } v, w \in E^0 \mbox{ with } v \neq w, \\
\sum_{e \in vE^1} s_es_e^* \leq p_v & \mbox{ for all } v \in E^0.
\end{cases} \]
Note that $\cT C^*(E)$ is unital if $E$ is finite. In this case, we have 
$1 = \sum_{v \in E^0} p_v$. 
For each $v \in E^0$, let 
\[ \delta_v = p_v - \sum_{vE^1}s_es_e^*. \]
Then, $\delta_v$ is a minimal projection of $\cT C^*(E)$. 
Define an action $\gauge \colon \bR \acts \cT C^*(E)$ by 
\[ \gauge_t (p_v) = p_v,\ \gauge_t(s_e) = e^{it}s_e \]
for each $t \in \bR$, $v \in E^0$, and $e \in E^1$. Such an action uniquely exists by the universal property of $\cT C^*(E)$. 
The action $\alpha$ is called the \emph{gauge action}. 

We investigate \emph{KMS$_\beta$ states} of the C*-dynamical system $(\cT C^*(E), \gauge)$ for $\beta \in \bR_+$. See \cite[Section 5.3]{BR2} for the basics of KMS states. In particular, a KMS$_\beta$ state is extremal if and only if it is a factor state by \cite[Theorem 5.3.30]{BR2}. The \emph{type} of a KMS$_\beta$ state $\varphi$ is the type of the factor generated by the GNS representation of $\varphi$. 

\subsection{KMS states of type I and partition functions of vertices} \label{ssec:typeI}
The structure of extremal KMS states of $\cT C^*(E)$ is determined in \cite[Theorem 5.3]{aHLRS15}. We summarize their results for KMS states of type I and type III separately. 

\begin{dfn}
Let $E$ be a finite graph. For a vertex $v \in E^0$, the funtion $Z_v \colon \bR \to (0,\infty]$ defined by
\[ Z_v(\beta) := \sum_{\nu \in E^*v} e^{-\beta |\nu|} \] 
is called \emph{the partition function of $v$}. Moreover, 
\[ \beta_v = \inf \{\beta \in \bR \mid Z_v(\beta) < \infty \} \in \{-\infty\} \cup [0,\infty] \]
is called the \emph{critical inverse temperature of $v$}. 
\end{dfn}

The partition function $Z_v$ is introduced in \cite[Equation 5.8]{CarLar}, and later an abstract characterisation is given in \cite{BT1}. 
The function $Z_v$ is the generating function of finite paths terminating at $v$. 
Since $E$ is finite, we can see that $\beta_v < \infty$ for all $v \in E^0$ (see also \cite[Theorem 4.11]{BT1}). 
We are interested only in vertices satisfying $\beta_v>0$, which is a necessary condition to have type III KMS states at the critical inverse temperature. Hence, we always restrict $Z_v$ to $\bR_+$ when we consider partition functions.

\begin{rmk} \label{rmk:diverge}
For a graph $E$ and a vertex $v \in E^0$, it is known that we have 
\[ \lim_{\beta \to \beta_v+0} Z_v(\beta) = \infty. \]
See \cite[Lemma 7.2]{aHLRS15}, for example. 
\end{rmk}

Extremal KMS states of type I are characterised as follows:

\begin{thm}[{\cite[Theorem 5.3 (b)]{aHLRS15}}] \label{thm:typeI}
Let $E$ be a finite graph. 
For any $\beta \in \bR_+$, there is a one-to-one correspondence between the set of extremal type I KMS$_\beta$ states of $(\cT C^*(E), \gauge)$ and the set of vertices $v \in E^0$ satisfying $\beta>\beta_v$. 
\end{thm}

\begin{dfn}
Let $E$ be a finite graph. 
For a vertex $v \in E^0$ and a real number $\beta>\beta_v$,  
the extremal type I KMS$_\beta$ state of $(\cT C^*(E), \gauge)$ corresponding to $v$ by Theorem \ref{thm:typeI} is denoted by $\varphi_{\beta,v}$. 
\end{dfn}

We introduce other sorts of generating functions, which play a fundamental role in this article.  

\begin{dfn}
Let $E$ be a finite graph, 
let $v,w \in E^0$, and let $C \in \pi(E)$. 
Define functions on $\bR_+$ by 
\[ 
Z_{w,v}(\beta) := \sum_{\nu \in wE^*v} e^{-\beta |\nu|},\ 
Z_v^C(\beta) := \sum_{w \in C} Z_{w,v}(\beta),  \mbox{ and }
Z_{w,v}^C(\beta) := \sum_{w' \in C} Z_{w,w'}(\beta)Z_{w',v}(\beta). 
\]
\end{dfn}
The functions $Z_{w,v}, Z_v^C, Z_{w,v}^C$ are generating functions of finite paths from $w$ to $v$, of finite paths from $C$ terminating at $v$, and of finite paths from $w$ to $v$ going through $C$ respectively. Thus, if $v \in C$, then $Z_v^C$ is the partition function of $v$ in the graph $E_C$. Moreover, if $v,w \in C$, then we have $Z_{w,v}=Z_{w,v}^C$. 

The KMS$_\beta$ state $\varphi_{\beta,v}$ is conceretely determined in \cite{aHLRS15} (see also \cite[Lemma 4.5]{BT1}) as follows: 

\begin{prop}[{\cite[Corollary 5.2]{aHLRS15}}] \label{prop:measure}
Let $E$ be a finite graph. 
For any $v \in E^*$ and $\beta>\beta_v$, we have 
\[ \varphi_{\beta,v}(p_w) 
=\frac{Z_{w,v}(\beta)}{Z_v(\beta)} \mbox{ and }
\varphi_{\beta,v}(\delta_v) = \frac{1}{Z_v(\beta)}. 
\]
In particular, $\varphi_{\beta,v}(p_w)=0$ if $wE^*v = \emptyset$. 
\end{prop}

\subsection{KMS states of type III and minimal components}

The Toeplitz algebra of a finite graph $E$ has a groupoid model $\cG_E$ with the base space $\cG_E^{(0)}=E^{\leq \infty}$ (see \cite{Pat}). It is shown in \cite[Proposition 2.1]{aHLRS13} that any KMS$_\beta$ state on $\cT C^*(E) \cong C^*(\cG_E)$ factors through the canonical expectation $C^*(\cG_E) \to C(E^{\leq \infty})$, which implies that there is an isomorphism of simplices between the simplex of KMS$_\beta$ states of $(\cT C^*(E),\gauge)$ and the $e^\beta$-conformal measures on $E^{\leq \infty}$. We say that a KMS$_\beta$ state $\varphi$ of $(\cT C^*(E),\gauge)$ is \emph{supported in $E^\infty$} if the corresponding meausre is supported in $E^\infty \subset E^{\leq \infty}$. Note that $\varphi$ is supported in $E^\infty$ if and only if $\varphi(\delta_v)=0$ for all $v \in E^0$.  
In this subsection, we summarize classification results of KMS states supported in $E^\infty$. 

For a finite graph $E$, a nonnegative vector $\harm = [\harm_v]_{v \in E^0} \in \bR^{E^0}$ is called a \emph{$\beta$-harmonic vector} if $\harm_v \geq 0$ for any $v \in E^0$ and $A_E\harm = e^\beta \harm$. The next theorem is due to Thomsen \cite{Th17}, which is a fundamental tool of the investigation of type III KMS states of $\cT C^*(E)$. 
\begin{thm}[{\cite[Theorem 2.7]{Th17}}] \label{thm:Thomsen}
Let $E$ be a finite graph. The cone of KMS$_\beta$ positive linear functionals of $(\cT C^*(E),\gauge)$ supported in $E^\infty$ is isomorphic to the cone of $\beta$-harmonic vectors via the map 
\[ \varphi \mapsto [\varphi(p_v)]_{v \in E^0}. \]
\end{thm}

Under the isomorphism in Theorem \ref{thm:Thomsen}, KMS states correspond to $\beta$-harmonic vectors $\harm \in \bR^{E^0}$ satisfying that $\|\harm\|_1=1$. For a vector $\harm \in \bR^{E^0}$ and a subset $S \subset E^0$, let $\harm_S := [\harm_v]_{v \in S} \in \bR^S$ be the restriction of $\harm$ to $S$. 

The next definition is introduced in \cite{BT1}. 

\begin{dfn}[{\cite[Definition 5.12]{BT1}}]
Let $E$ be a finite graph. 
A component $C \in \pi(E)$ is called \emph{minimal} if $\rho(A_{C'}) < \rho(A_C)$ holds for any $C' \in \pi(E)$ satisfying $C' \neq C$ and $C'E^*C \neq \emptyset$. The set of all minimal components $C$ with $\rho(A_C)>1$ is denoted by $\pmc(E)$. 
\end{dfn}

The classification of KMS states concentrated in $E^\infty$ is due to an Huef--Laca--Raeburn--Sims \cite[Theorem 5.3]{aHLRS15}.  
Here, we need a technical variant of their results. 
The next theorem follows from \cite[Theorem 4.3]{aHLRS15} and its proof. We also include a direct self-contained proof of the next theorem in Appendix \ref{sec:appendix} for reader's covenience. 
For a finite graph $E$ and $C \in \pmc(E)$, let $\beta_C = \log \rho(A_C)$. 

\begin{thm}\label{thm:aHLRS15}
Let $E$ be a finite graph. 
Let $\beta>0$, and let
\[ K = \bigcup \{C \in \pmc(E) \colon \beta_C = \beta\}. \]
Let $\harm^0 \in \bR^K$ be a $\beta$-harmonic vector of $E_K$. Then, the following hold: 
\begin{enumerate}
\item For any $\beta$-harmonic vector $\harm \in \bR^{E^0}$ of $E$, the restriction $\harm_K \in \bR^K$ is also a $\beta$-harmonic vector of $E_K$. \label{item:restriction}
\item The $\beta$-harmonic vector $\harm^0$ uniquely extends to a $\beta$-harmonic vector $\harm \in \bR^{E^0}$ of $E$. \label{item:uniqueness}
\item The $\beta$-harmonic vector $\harm$ in \eqref{item:uniqueness} is determined by 
\[ 
\begin{cases}
\harm_K = \harm^0, \\
\harm_D = (e^\beta I - A_{D})^{-1} A_{D, K} \harm^0, \\
\harm_{E^0\setminus S_K}=0, 
\end{cases}
\]
where $D = S_K \setminus K$. \label{item:extension}
\end{enumerate}
\end{thm}

Combining \eqref{item:restriction} and \eqref{item:uniqueness} of Theorem \ref{thm:aHLRS15}, the map $\harm^0 \mapsto \harm$ induces an affine isomorphism between the simplex of $\beta$-harmonic vectors $\harm^0$ of $E_K$ with $\|\harm\|_1=1$ and  the simplex of KMS$_\beta$ states of $\cT C^*(E)$ supported in $E^\infty$. 
Let $C_1,\cdots,C_n$ be the listing of minimal components of $E$ with $\beta_{C_i}=\beta$. Then, we have
\begin{equation} \label{eqn:diagonal} 
A_K = \begin{bmatrix} A_{C_1} & & \\ & \ddots & \\ & & A_{C_n}\end{bmatrix}. 
\end{equation}
In addition, by the Perron--Frobenius theorem (see \cite[Theorem 1.5]{Sen}), $e^\beta$ is a simple eigenvalue of $A_{C_i}$ for each $i=1,\cdots,n$. Hence, the extremal points of the simplex of $\beta$-harmonic vectors $\harm^0$ of $E_K$ with $\|\harm\|_1=1$ coincides with the Perron--Frobenius eigenvectors of $A_{C_i}$ for some $i$ satisfying the appropriate normalisation. Therefore, there is a one-to-one correspondence between $\pmc(E)$ and the union of the set of extremal KMS$_\beta$ states of $\cT C^*(E)$ supported in $E^\infty$, where the union is taken for all $\beta>0$. This consequence is essentially the same as \cite[Theorem 5.3 (c)]{aHLRS15} (see also \cite[Theorem 5.16]{BT1}). 

\begin{dfn}
Let $E$ be a finite graph. 
For $C \in \pmc(E)$, let $\psi_C$ denote the extremal KMS$_{\beta_C}$ state corresponding to $C$ from Theorem \ref{thm:aHLRS15}. 
\end{dfn}

Concretely, the KMS state $\psi_C$ is obtained as follows: we first apply Theorem \ref{thm:aHLRS15} to the $\beta$-harmonic vector $\tilde{\harm}^0$ of $E_K$ satisfying that $\tilde{\harm}_C^0$ is a Perron--Frobenius eigenvector of $A_C$ and $\tilde{\harm}_{K\setminus C}^0 =0$, so that we extend $\tilde{\harm}^0$ to a $\beta$-harmonic vector $\tilde{\harm}$ of $E$. Then, $\psi_C$ is the KMS state of $\cT C^*(E)$ corresponding to the $\beta$-harmonic vector $\harm = \tilde{\harm}/\| \tilde{\harm} \|_1$ of $E$ from Theorem \ref{thm:Thomsen}. The KMS state $\psi_C$ is known to be type III by \cite[Theorem 5.17]{BT1}, and the set of all extremal type III KMS states of $\cT C^*(E)$ coincides with $\{\psi_C \colon C \in \pmc(E)\}$. Hence, we sometimes refer to $\psi_C$ as extremal type III KMS states instead of referring to them as extremal KMS states supported in $E^\infty$. 

Finally, the next definition will be used to describe the relation between vertices and components.  
\begin{dfn}
Let $E$ be a finite graph, and let $v \in E^0$ be a vertex with $\beta_v>0$. Define 
\[ \crit{v}(E) = \{ C \in \pmc(E) \colon CE^*v \neq \emptyset,\ \beta_C = \beta_v\}. \]
\end{dfn}

\section{Main Results} \label{sec:main}

\subsection{General facts of KMS states of $\cT C^*(E)$} \label{ssec:decomp}
We begin with making several observations on convex decompositions of KMS states of $\cT C^*(E)$ and limits of KMS states of $\cT C^*(E)$. 
First, we establish a way to calculate the coefficients in the convex decomposition of a KMS state from the values of vertex projections belonging to minimal components. 

\begin{prop} \label{prop:coefficients}
Let $E$ be a finite graph, and let $\varphi$ be a KMS$_\beta$ state of $\cT C^*(E)$ supported in $E^\infty$. 
Let $K \subset E^0$ be the subset from Theorem \ref{thm:aHLRS15}. 
For each $C \in \pi(E_K)$, 
define $\harm^C = [\harm^C_w]_{w \in S_C} \in \bR^{S_C}$ by 
\[ \harm_w^C = \varphi(p_w) \mbox{ for } w \in C,\ 
\harm_{D}^C = (e^\beta I - A_{D})^{-1} A_{D, C} \harm_C^C, \]
where $D=S_C \setminus C$. 
Then, we have 
\[ \varphi = \sum_{C \in \pi(E_K)} \| \harm^C\|_1 \psi_C. \] 
\end{prop}

\begin{proof}
First, note that $\pi(E_K)=\{C \in \pmc(E) \colon \beta_C=\beta\}$. 
By Theorem \ref{thm:aHLRS15} \eqref{item:restriction} combined with Theorem \ref{thm:Thomsen}, the vector $[\varphi(p_w)]_{w \in K} \in \bR^K$ is a $\beta$-harmonic vector of $E_K$. For each $C$, the vector $\harm^{0,C}$ defined by 
\begin{equation} \label{eqn:harmforC} 
\harm^{0,C}_w = \begin{cases} \varphi(p_w) & \mbox{if } w \in C \\ 0 & \mbox{if } w \in K \setminus C, \end{cases} 
\end{equation}
is also a $\beta$-harmonic vector of $E_K$, since $A_K$ is a block diagonal matrix as in Equation \eqref{eqn:diagonal}.   

For each $C \in \pmc(E_K)$, apply Theorem \ref{thm:aHLRS15} \eqref{item:uniqueness} to the $\beta$-harmonic vector $\harm^{0,C}$ of $E_K$ defined in Equation \eqref{eqn:harmforC} 
to obtain an extension of $\harm^{0,C}$ to a $\beta$-harmonic vector of $E$, which is also denoted by $\harm^{0,C}$, and a KMS$_\beta$ positive linear functional $\varphi_C$ supported in $E^\infty$ corresponding to $\harm^{0,C}$. 
Then, we have 
\begin{equation} \label{eqn:DecompIntermed}
 \varphi = \sum_{C \in \pi(E_K)} \varphi_C 
 \end{equation}
by the uniqueness in Theorem \ref{thm:aHLRS15} \eqref{item:uniqueness}, 
since the values of $p_w$ in both the left and right hand side are the same for all $w \in K$. 

Moreover, for each $C \in \pi(E_K)$, we have 
\begin{equation} \label{eqn:haRmonic} 
\begin{cases}
\varphi_C(p_w) = \harm^C_w & \mbox{ for } w \in S_C, \\
\varphi_C(p_w)=0 & \mbox{ for } w \in E^0\setminus S_C.
\end{cases}
\end{equation}
In order to see this, let $D_1=(S_K \setminus S_C) \setminus (K \setminus C)$, and let $D_0 = D_1 \cup D = S_K \setminus K$. Then, since there are no paths between $D_1$ and $D$, and no paths from $D_1$ to $C$, we have 
\[ A_{D_0,K} = \begin{bmatrix} A_{D,C} & A_{D,K\setminus C} \\ 0 & A_{D_1, K\setminus C} \end{bmatrix},\ 
A_{D_0} = \begin{bmatrix} A_{D} & 0 \\ 0 & A_{D_1} \end{bmatrix}. \]
Hence, by Theorem \ref{thm:aHLRS15} \eqref{item:extension}, we have
\[ \harm_{D_0}^{0,C} = (e^{\beta}I - A_{D_0})^{-1}A_{D_0,K}\harm_K^{0,C} 
= \begin{bmatrix} (e^\beta I - A_D)^{-1} A_{D,C} \harm_C^{0,C} \\ 0 \end{bmatrix}
= \begin{bmatrix} \harm_D^C \\ 0 \end{bmatrix}, \]
which implies \eqref{eqn:haRmonic}. 

Equation \eqref{eqn:haRmonic} implies $\varphi_C(1) = \|\harm^C\|_1$. 
By the uniqueness in Theorem \ref{thm:aHLRS15} \eqref{item:uniqueness} and the uniqueness of the Perron--Frobenius eigenvector up to positive constants, we have $\varphi_C = \|\harm^C\|_1 \psi_C$, which completes the proof. 
\end{proof}

\begin{rmk}
For each $w \in S_K \setminus K$, $\harm^C_w$ is not necessarily equal to $\varphi(p_w)$ in Proposition \ref{prop:coefficients}. In fact, from Equation \eqref{eqn:DecompIntermed} implies that 
\begin{equation} \label{eqn:DecompCoeff}
 \sum_{C \in \pi(E_K)} \harm^C_w = \varphi(p_w) 
\end{equation}
for all $w \in E^0$ (we regard $\harm^C_w=0$ if $w \not\in S_C$ here). In other words, the value $\varphi(p_w)$ is distributed to coefficients at some proportion. In addition, we can directly see that $\sum_{C \in \pmc(E)} \|\harm^C\|_1=1$ from Equation \eqref{eqn:DecompCoeff}. 
Moreover, in Proposition \ref{prop:coefficients}, we can see that if $C \in \pi(E_K)$ and $\varphi(p_w)=0$ for some $w \in C$ (hence, for all $w \in C$), then  $\harm^C=0$. 
\end{rmk}

\begin{cor} \label{cor:coefficients}
Let $E$ be a finite graph, and let $v \in E^0$ be a vertex such that $\beta_v>0$. Suppose the weak*-limit $\displaystyle \varphi = \lim_{\beta \to \beta_v+0} \varphi_{\beta,v}$ exists. Then, $\varphi$ is supported in $E^\infty$. For $C \in \crit{v}(E)$, let $\harm^C \in \bR^{S_C}$ be the vector from Proposition \ref{prop:coefficients} applied to $\varphi$. Then, we have 
\[ \varphi = \sum_{C \in \crit{v}(E)} \| \harm^C\|_1 \psi_C. \] 
Moreover, we have  
\[ \harm^C_w = \lim_{\beta \to \beta_v+0} \frac{Z_{w,v}^C(\beta)}{Z_v(\beta)} \]
for any $w \in S_C$. 
\end{cor}
\begin{proof}
First, Proposition \ref{prop:measure} combined with Remark \ref{rmk:diverge} implies that $\varphi$ is supported in $E^\infty$. 
Let $K \subset E^0$ be the subset from Theorem \ref{thm:aHLRS15}. 
By Proposition \ref{prop:measure}, we have $\varphi_{\beta,v} (p_w) = 0$ for all $\beta>\beta_v$ and hence $\varphi(p_w)$=0, if $w \in E^0$ belongs to a component $C \in \pi(E_K) \setminus \crit{v}(E)$. Hence, $\harm^C$ in Proposition \ref{prop:coefficients} is zero for all $C \in \pi(E_K) \setminus \crit{v}(E)$, which implies the second claim. Fix $C \in \crit{v}(E)$, and we prove the last claim. For $w \in C$, we have 
\[ \harm^C_w = \varphi(p_w) = \lim_{\beta \to \beta_v+0} \frac{Z_{w,v}(\beta)}{Z_v(\beta)} \]
by Proposition \ref{prop:measure}, which implies that the last claim holds for any $w \in C$. Let $w \in D$, where $D=S_C \setminus C$ as in Proposition \ref{prop:coefficients}. 
By the definition of $\harm^C$, we have 
\[
\harm_{D}^C = (e^{\beta_v} I - A_{D})^{-1} A_{D, C} \harm_C^C 
= \lim_{\beta \to \beta_v+0} e^{-\beta}\sum_{k=0}^\infty e^{-k\beta} A_D^k A_{D,C} \harm_C^C. 
\]
Note that $\rho(A_D) <e^{\beta_v}$ by the minimality of $C$ combined with Equation \eqref{eqn:seneta}. 
Hence, we can see that
\[ \harm_{w} = \lim_{\beta \to \beta_v+0} \sum_{w' \in C} Z_{w,w'}(\beta) \harm_{w'}^C 
= \lim_{\beta \to \beta_v+0} \sum_{w' \in C} \frac{Z_{w,w'}(\beta)Z_{w',v}(\beta)}{Z_v(\beta)}
=\lim_{\beta \to \beta_v+0} \frac{Z_{w,v}^C(\beta)}{Z_v(\beta)}. \]
Therefore, the last claim holds for all $w \in S_C$. 
\end{proof}

The assumption of the existence of the limit in Corollary \ref{cor:coefficients} is in fact always satisfied, which will be proven in Theorem \ref{thm:main}. 
The next proposition allows us to reduce a problem of the convergence of KMS states to that of the convergence in $\bR$-valued functions on $E^0$, which will be used at the final step of the proof of Theorem \ref{thm:main}.  

\begin{prop} \label{prop:GeneralLimit}
Let $\beta_0, \beta_1 \in \bR_+$ with $\beta_1 > \beta_0$. For each $\beta \in (\beta_0,\beta_1]$, let $\varphi_{\beta}$ be a KMS$_\beta$ state of $\cT C^*(E)$. For each $v \in E^0$, let $\harm_v \in [0,1]$. Suppose 
\[ \lim_{\beta \to \beta_0+0} \varphi_\beta(p_v) = \harm_v, \mbox{ and } \lim_{\beta \to \beta_0+0} \varphi_\beta(\delta_v) = 0 \]
for all $v \in E^0$. Then, $\varphi_\beta$ converges to a KMS$_{\beta_0}$ state supported in $E^\infty$ 
as $\beta \to \beta_0+0$. 
\end{prop}
\begin{proof}
By the KMS condition, we have 
\begin{align*}
\varphi_\beta (p_w) = \varphi_\beta (\delta_w) + \sum_{e \in wE^1} \varphi_\beta (s_es_e^*) 
&= \varphi_\beta (\delta_w) + \sum_{e \in wE^1} e^{-\beta} \varphi_\beta (p_{r(e)}) \\
&= \varphi_\beta (\delta_w) + e^{-\beta} \sum_{v \in E^0} |wE^1v| \varphi_\beta (p_v) 
\end{align*}
for any $w \in E^0$. By taking the limit, we have $\harm = e^{-\beta_0} A_E \harm$, which implies that $\harm$ is a $\beta_0$-harmonic vector. By Theorem \ref{thm:Thomsen}, there exists a KMS$_{\beta_0}$ state $\varphi$ corresponding to $\harm$. 
For each $\mu=e_0\cdots e_n \in E^*$, let $s_\mu := s_{e_0} \cdots s_{e_n} \in \cT C^*(E)$. Let $\cS$ be the linear span of $s_\mu s_\nu^*$ for all $\mu, \nu \in E^*$ with $r(\mu)=r(\nu)$ in $\cT C^*(E)$. Since all KMS states of $\cT C^*(E)$ factors through the canonical diagonal by \cite[Proposition 2.1]{aHLRS13}, we have $\varphi_\beta (s_\mu s_\nu^*) = 0 = \varphi(s_\mu s_\nu^*)$ if $\mu \neq \nu$. Moreover, by the KMS condition, 
\[ \varphi_\beta(s_\mu s_\mu^*) = e^{-\beta |\mu|} \varphi_\beta (p_{r(\mu)}) 
\longrightarrow e^{-\beta_0 |\mu|} \varphi (p_{r(\mu)}) = \varphi(s_\mu s_\mu^*). \]
Hence, for any $a \in \cS$, $\varphi_\beta(a)$ converges to $\varphi(a)$. 
It is well-kown (and easy to see) that $\cS$ is dense in $\cT C^*(E)$ (see \cite[Lemma 2.4]{FowRae}, for example). Hence, $\varphi_\beta$ converges to $\varphi$ in the weak*-topology. 
\end{proof}

\subsection{An Equivalence relation of power series}
The key point of the proof of Theorem \ref{thm:main} is the decomposition of each generating function to a product of generating functions localised at each components, up to an equivalence relation (see Definition \ref{def:product} and Lemma \ref{lem:summation}). Here, we prepare an equivalence relation of power series that is suitable for our calculation. We define the equivalence relation in a slightly general setting. 
Similar notions of the equivalence relation could be found in other literatures.  

\begin{dfn} \label{def:functions}
Let $\cF_0$ denote the set of all continuous monotone increasing functions $F=F(x) \colon \bR_+ \to (0, \infty]$ such that $F(x) < \infty$ for some $x>0$. 
For $F \in \cF_0$, let 
\[ x_F = \sup \{ x \in \bR_+ \colon F(x) < \infty \} \in (0,\infty]. \] 
\end{dfn}

By continuity, we have $F(x_F)=\infty$ if $x_F < \infty$. 
The set $\cF_0$ is closed under addition and multiplication. 

\begin{dfn} \label{def:equivalence} Define an equivalence relation $\approx$ on $\cF_0$ as follows: 
 For $F,G \in \cF_0$, we write $F \approx G$ if $x_F = x_G$ and the limit 
\[ \lim_{x \to x_F-0} \frac{F(x)}{G(x)} \]
exists and is finite nonzero. 
For $F \in \cF_0$, the equivalence class of $F$ is denoted by $[F]$. 
\end{dfn}

If $F,G \in \cF_0$ are analytic functions that have $n$-th poles at $x=x_F=x_G$, then $F \approx G$. We use the equivalence relation $\approx$ only in this situation. 
We leave to the reader to prove that $\approx$ is indeed an equivalence relation, and the proof of the next lemma. 

\begin{lem} \label{lem:PropertiesEquiv} 
For $F,F_1,F_2,G_1,G_2 \in \cF_0$, the following hold:
\begin{enumerate}
\item If $x_F < x_G$, then $F+G \approx F$. \label{equiv:addition} 
\item If $x_F < x_G$, then $FG \approx F$. \label{equiv:multiplication1}
\item If $F_1 \approx G_1$ and $F_2 \approx G_2$, then $F_1F_2 \approx G_1G_2$. \label{equiv:multiplication2}
\item For any continuous monotone increasing function $f \colon \bR_+ \to (0,\infty)$, we have $fF \approx F$. \label{equiv:scalar}
\end{enumerate}
\end{lem}

\begin{dfn}
For $\alpha \in \bR_+$, let 
\[ \Phi_{\alpha}(x) = \frac{1}{1-\alpha^{-1}x}. \]
In addition, define
\[ \cF_1 = \{F \in \cF_0 \colon F \approx \Phi_{\alpha}^n \mbox{ for some } n \in \bN \mbox{ and } \alpha \in \bR_+\} \]
and let 
\[ \cF = \cF_1/\approx \ = \bigl\{ [\Phi_{\alpha}^n] \colon n \in \bN, \alpha \in \bR_+ \bigr\}. \]
\end{dfn}

Note that if $F \in \cF_0$ is an analytic function that has a $n$-th pole at $x=x_F$, then $F \in \cF_1$ and $[F]=[\Phi_{x_F}^n]$. 

\begin{lem} \label{lem:PropertiesEquiv2}
Let $F_1,F_2,G_1,G_2 \in \cF_1$ and suppose $F_1 \approx G_1$, $F_2 \approx G_2$. Then, we have $F_1+F_2, G_1+G_2 \in \cF_1$ and $F_1+F_2 \approx G_1+G_2$. 
\end{lem}
\begin{proof}
Suppose $F_1 \approx G_1 \approx \Phi_{\alpha}^n$ and $F_2 \approx G_2 \approx \Phi_{\beta}^m$. Then, $x_{F_1}=x_{G_1}=\alpha$ and $x_{F_2}=x_{G_2}=\beta$. If $\alpha < \beta$, then by \eqref{equiv:addition} of Lemma \ref{lem:PropertiesEquiv}, we have $F_1+F_2 \approx F_1$ and $G_1+G_2 \approx G_1$, so that $F_1+F_2 \approx G_1+G_2$. The case $\beta < \alpha$ is similar. Hence, we assume $\alpha=\beta$. In addition, we may assume $n \leq m$. Then, we can see that $F_1+G_1 \approx \Phi_{\alpha}^m \approx F_2+G_2$, which completes the proof. 
\end{proof}

\begin{rmk} \label{rmk:totalorder}
By Lemma \ref{lem:PropertiesEquiv} \eqref{equiv:multiplication2} and Lemma \ref{lem:PropertiesEquiv2}, we define the addition and the multiplication on $\cF$ by those induced from $\cF_1$. 
In addition, $\cF$ has the total order $\preceq$ defined by $[F] \preceq [G]$ if $x_F \geq x_G$ and 
$[F]+[G]=[G]$. Note that $[\Phi_\alpha]^n \preceq [\Phi_\beta]^m$ if and only if either $\alpha>\beta$, or $\alpha = \beta$ and $n \leq m$. 

If $F,G \in \cF_1$ satisfies $[F] \prec [G]$, then we can see that 
\[ \lim_{x \to x_G-0} \frac{F(x)}{G(x)} = 0. \]
Therefore, for any $F,G \in \cF_1$, the limit 
\[ \lim_{x \to x_0-0} \frac{F(x)}{G(x)} \]
always exists, where $x_0 = \min \{x_F, x_G\}$. 
\end{rmk}

We need observations related to the Perron--Frobenius theorem. 
Recall that for a nonnegative irreducible matrix $A \in M_n(\bR)$, the \emph{Perron projection of $A$} is the (not necessarily orthogonal)  projection $\xi^l (\xi^r)^*$ on $\bR^n$, where $\xi^l, \xi^r \in \bR^n$ are left and right Perron--Frobenius eigenvectors of $A$ respectively, satisfying that $\langle \xi^r, \xi^l \rangle=1$. Note that the standard inner product of $\bR^n$ is considered here. The next lemma is possibly well-known. 

\begin{lem} \label{lem:simplepole}
Let $A \in M_n(\bZ)$ be a nonnegative irreducible matrix. Let $\lambda = \rho(A)$. Then, $(1-zA)^{-1}$ has a simple pole at $z=\lambda^{-1}$. 
That is, the limit
\[ \lim_{z \to \lambda^{-1}} \frac{1-z\lambda}{1-zA} \]
exists and nonzero. Moreover, this limit coincides with the Perron projection of $A$. 
\end{lem}

\begin{proof}
By the Perron--Frobenius theorem \cite[Theorem 1.5]{Sen}, $\lambda$ is a simple root of the characteristic polynomial of $A$. Hence, there exists $R \in GL_n(\bC)$ and $A_0 \in M_{n-1}(\bC)$ such that 
\[ R^{-1}AR = \begin{bmatrix}  \lambda & 0 \\ 0 & A_0 \end{bmatrix}.  \]

Then, we can see that the Perron projection of $A$ is equal to $R \begin{bmatrix} 1 & 0 \\ 0 & 0 \end{bmatrix} R^{-1}$. 
For $z \in \bC \setminus \Sp(A)$, we have
\[ (1-zA)^{-1} = R\begin{bmatrix}  (1-z\lambda)^{-1} & 0 \\ 0 & (1-zA_0)^{-1} \end{bmatrix} R^{-1}. \]
Since $\lambda \not\in \Sp(A_0)$, we have
\[ \lim_{z \to \lambda^{-1}} \frac{1-z\lambda}{1-zA} = R \begin{bmatrix} 1 & 0 \\ 0 & 0 \end{bmatrix} R^{-1}, \]
which completes the proof. 
\end{proof}

\begin{prop} \label{prop:simplepole}
Let $E$ be a finite strongly connected graph satisfying $\rho(A_E)>0$. Let $P$ be the Perron projection of $A_E$. Let $\{\xi_v\}_{v \in E^0}$ be the  orthonormal basis of $\bR^{E^0}$ with respect to the standard inner product. Then, for any $v,w \in E^0$, 
\[ Z_{w,v}(-\log z) = \sum_{\nu \in wE^*v} z^{ |\nu|} \]
has a simple pole at $z = \rho(A_E)^{-1}$ with residue $-\lambda^{-1} \langle P \xi_v, \xi_w \rangle$. 
\end{prop}
\begin{proof}
Let 
$\lambda=\rho(A)$, and fix $v,w \in E^0$. 
Then, the function $Z_{w,v}(-\log z)$ is equal to the $(w,v)$-entry of 
\[ \sum_{k=0}^\infty z^kA^k = \frac{1}{1-zA}. \]
Hence, $Z_{w,v}(-\log z)$ has an analytic continuation to a holomorphic function on a neighbourhood of $\lambda^{-1}$ in $\bC$ except for $z= \lambda^{-1}$. 
Moreover, by Lemma \ref{lem:simplepole}, 
\begin{align*}
\lim_{z \to \lambda^{-1}} (z-\lambda^{-1})Z_{w,v}(-\log z) &= 
-\lambda^{-1} \lim_{z \to \lambda^{-1}} (1-z\lambda)Z_{w,v}(-\log z) \\
&= -\lambda^{-1} \left\langle \lim_{z \to \lambda^{-1}} \frac{1-z\lambda}{1-zA}\xi_v, \xi_w \right\rangle
= -\lambda^{-1} \langle P \xi_v, \xi_w \rangle. 
\end{align*}

It remains to show that $\langle P \xi_v, \xi_w \rangle >0$. 
Let $\xi^l, \xi^r \in \bR^{E^0}$ be the left and right Perron--Frobenius eigenvectors of $A_E$ respectively, satisfying that $\langle \xi^r, \xi^l \rangle=1$. By definition, we have $\langle \xi^l, \xi_{w'} \rangle >0$ and $\langle \xi^r, \xi_{w'} \rangle >0$ for all $w' \in E^0$. Since $P\eta = \langle \eta, \xi^r \rangle \xi^l$ for any $\eta \in \bR^{E^0}$, we have
\[ \langle P \xi_v, \xi_w \rangle = \langle \xi_v, \xi^r \rangle \langle \xi^l, \xi_w \rangle >0, \]  
which completes the proof. 
\end{proof}

\begin{rmk}
The graph $E_0$ satisfying $|E^0_0|=1$ and $|E^1_0|=0$ is a strongly connected graph by definition. For a non-empty strongly connected graph $E$, $\rho(A_E)=0$ if and only if $E$ is isomorphic to $E_0$. 
\end{rmk}

\subsection{Graphs of strongly connected components}

From a finite graph $E$ and a component $C_0$, we define finite graphs by identifying all vertices in each components. 

\begin{dfn}
For a finite graph $E$ and a component $C_0 \in \pi(E)$, we define a finite graph $B(E,C_0)$ as follows: The vertex set is
\[ B(E,C_0)^0 = \{C \in \pi(E) \colon CE^*C_0 \neq \emptyset \}. \]
We put an edge from $C' \in B(E,C_0)^0$ to $C \in B(E,C_0)^0$ if there exists $e \in E^1$ such that $s(e) \in C'$ and $r(e) \in C$. The source and range maps of $B(E,C_0)$ are denoted by $S$ and $R$ respectively. 
\end{dfn}

For simplicity, $B(E,C_0)$ is just denoted by $B$ when $E$ and $C_0$ are obvious. 
By definition, $C_0$ is the unique sink in $B(T,C_0)$, which is connected from all vertices in $B$. For any $C,C' \in B^0$, we have $|C'B^1C| \in \{ 0,1 \}$. In addition, $B$ does not contain loops. 
The graph $B$ and maximal paths defined as below will be used to formulate Theorem \ref{thm:main}. 

\begin{dfn} \label{def:maximal}
Let $E$ be a finite graph and let $v \in E^0$. Let $C_0 \in \pi(E)$ be the component containing $v$. 
\begin{enumerate} 
\item A component $C \in B(E,C_0)^0$ is called \emph{critical} if $\rho(A_C)=e^{\beta_{v}}$. \label{item:critical}
\item For each path $\mu = e_0\cdots e_n \in B^*C_0$, let \label{item:path}
\[ \critical(\mu) = \{ C \in \{S(e_i)\}_{i=0}^n \cup \{C_0\}\ \colon \rho(A_C)=e^{\beta_{v}}\}. \]
be the set of critical components in $\mu$. 
A path $\mu \in B^*C_0$ is called a \emph{maximal path to $C_0$} if 
\[ |\critical(\mu)| = \max \{ |\critical(\nu)| \colon \nu \in B^*C_0\}. \]
\end{enumerate}
\end{dfn}

Note that the critical component in Definition \ref{def:maximal} is a different notion from those defined in \cite{aHLRS15}. 
The information of the number of edges between components is not contained in the definition of $B(E,C_0)$, because it will turn out to be unnecessary. However, we technically need another graph containing that information in an intermediate step. 
\begin{dfn}
For a finite graph $E$ and a component $C_0 \in \pi(E)$, we define a finite graph $\tilde{B}(E,C_0)$ as follows: The vertex set is
\[ \tilde{B}(E,C_0)^0 = B(E,C_0)^0 = \{C \in \pi(E) \colon CE^*C_0 \neq \emptyset \}, \]
and the edge set is
\[ \tilde{B}(E,C_0)^1 = \{ e \in E^1 \colon s(e) \mbox{ and } r(e) \mbox{ belong to different components} \}. \]
The source and range maps of $\tilde{B}(E,C_0)$ are denoted by $S$ and $R$ respectively. For each $e \in \tilde{B}(E,C_0)^1$ and $C \in \tilde{B}(E,C_0)^0$, define 
$S(e)=C$ if $s(e) \in C$ in $E$, and $R(e)=C$ if $r(e) \in C$ in $E$. 
In addition, let $\varepsilon_{E,C_0} \colon \tilde{B}(E,C_0)^1 \to B(E,C_0)^1$ be the unique surjection preserving the source and range maps. 
\end{dfn}

As before, we omit $E$ and $C_0$ when they are obvious. 
We always identify $\tilde{B}^1$ with the subset of $E^1$.  
The graph $\tilde{B}(E,C_0)$ will only be used to show that we can indeed ignore the number of edges between components (see Lemma \ref{lem:summation}). 

Until the end of this subsection, we fix a finite graph $E$ and a component $C_0 \in \pi(E)$.  

\begin{dfn} \label{dfn:Btilde}
Let $v \in C_0$, $C \in \tilde{B}(E,C_0)^0$ and $w \in C$. 
For each $\tilde{\mu} = e_1e_2 \cdots e_n \in C\tilde{B}^*C_0$, define 
\[ X(\tilde{\mu},w,v) = \left\{ \nu_1 e_1 \nu_2 e_2 \cdots \nu_n e_n \nu_{n+1} \in E^1 \colon 
\begin{array}{c} \nu_1 \in wE^*s(e_1),\ \nu_{n+1} \in r(e_n)E^*v, \\
\nu_i \in r(e_{i-1})E^*s(e_i) \mbox{ for } i=2,\cdots,n
\end{array} \right\}. \]
\end{dfn}

Note that the source and range maps in Definition \ref{dfn:Btilde} are those of $E$ (not of $\tilde{B}$). In addition, 
each $\nu_i$ is a path in $E_{S(e_i)}$ for $i=1,\cdots,n$, and $\nu_{n+1}$ is a path in $E_{C_0}$ in the definition of $X(\tilde{\mu},w,v)$. 

\begin{lem} \label{lem:decomposition}
Let $v \in C_0$, $C \in \tilde{B}(E,C_0)^0$ and $w \in C$. Then, we have 
\begin{equation} \label{eqn:union}
 wE^*v = \bigsqcup_{\tilde{\mu} \in C\tilde{B}^*C_0} X(\tilde{\mu},w,v)  = \bigsqcup_{\mu \in CB^*C_0} \bigsqcup_{\tilde{\mu} \in \varepsilon^{-1}(\mu)} X(\tilde{\mu},w,v).  
\end{equation}
\end{lem}
\begin{proof}
Let $\nu \in wE^*v$ be a finite path. Then, $\nu$ can be uniquely written as 
\begin{equation} \label{eqn:paths}
\nu = \nu_1 e_1 \nu_2 e_2 \cdots \nu_n e_n \nu_{n+1}, 
\end{equation}
where $\nu_i$ is a path whose source and range belong to the same component, and $e_i$ is an edge whose source and range belong to different components. Then $\tilde{\mu}=e_1e_2\cdots e_n$ is a path in $\tilde{B}$, and we can see that $\nu \in X(\tilde{\mu},w,v)$. Moreover, the uniqueness of the decomposition in Equation \eqref{eqn:paths} implies that the unions in Equation \eqref{eqn:union} are indeed disjoint. 
\end{proof}

\subsection{The Main Theorem}

The purpose of this section is to prove the next theorem. 

\begin{thm} \label{thm:main}
Let $E$ be a finite graph. Let $v \in E^0$ be a vertex with $\beta_v>0$, and let $C_0 \in \pi(E)$ be the component containing $v$. Then, the following hold: 
\begin{enumerate}
\item The weak*-limit $\displaystyle \varphi_v :=\lim_{\beta \to \beta_v+0} \varphi_{\beta,v}$ exists. \label{main:existence}
\item Let \label{main:convex}
\begin{equation} \label{eqn:conclusion}
 \varphi_v = \sum_{C \in \crit{v}(E)} \lambda_C \psi_C 
\end{equation}
be the decomposition of $\varphi_v$ into a convex combination of extremal KMS$_\beta$ states. Then, $\lambda_C>0$ if and only if there exists a maximal path $\mu \in CB^*C_0$. 
\end{enumerate}
\end{thm}

By Corollary \ref{cor:coefficients}, the KMS state $\varphi$ in Theorem \ref{thm:main} can be decomposed into a convex combination as in Equation \eqref{eqn:conclusion}. 
Theorem \ref{thm:main} gives the complete description of components that have actual contribution to the convex combination. 
In order to determine the coefficients, Corollary \ref{cor:coefficients} is applicable. In Section \ref{sec:example}, we visit examples in which we demonstrate  concrete calculations of coefficients. 

\begin{rmk}
In the setting of Theorem \ref{thm:main}, if $C_0$ is minimal, then $\crit{v}(E)=\{C_0\}$, so that we have $\varphi_v = \psi_{C_0}$. In this case, the path of length $0$ at $C_0$ is a maximal path in $C_0 B^*C_0$, since $C_0$ is the unique critical component in $B$. 
Hence, Theorem \ref{thm:main} trivially holds in this case. 
\end{rmk}

Now we start proving Theorem \ref{thm:main}. Until the end of this subsection, we fix a finite graph $E$ and a vertex $v \in E^0$ with $\beta_v>0$. In addition, let $x_v = e^{-\beta_v}$, and let $C_0 \in \pi(E)$ be the component containing $v$. 

\begin{prop} \label{prop:partfuncformula}
Let $C \in \tilde{B}(E,C_0)^0$ and $w \in C$. For $\tilde{\mu}=e_1e_2\cdots e_n \in C\tilde{B}^*C_0$, let 
\[ Z_{\tilde{\mu}}(\beta) = \sum_{\nu \in X(\tilde{\mu},w,v)} e^{-\beta |\nu|}. \]
Then, we have 
\[ Z_{\tilde{\mu}}(\beta) = Z_{w,s(e_1)}(\beta) e^{-\beta} Z_{r(e_1), s(e_2)}(\beta) e^{-\beta} \cdots e^{-\beta} Z_{r(e_n),v}(\beta) \]
for $\tilde{\mu}=e_0e_1\cdots e_n \in C\tilde{B}^*C_0$. Moreover, we have 
\[
Z_{w,v}(\beta) = \sum_{\mu \in CB^*C_0} \sum_{\tilde{\mu} \in \varepsilon^{-1}(\mu)} Z_{\tilde{\mu}}(\beta). 
\]
\end{prop}
\begin{proof}
The first claim follows from the definition of $X(\tilde{\mu}, w, v)$, and the second claim follows 
from Lemma \ref{lem:decomposition}.  
\end{proof}

\begin{cor} \label{cor:generatingfunctions}
Let $Z=Z_{w,v},\ Z_v,\ Z_v^C$, or $Z_{w,v}^C$ for some $w,v \in E^0$ and $C \in \pi(E)$. Then, we have $Z \circ (-\log) \in \cF_1$. 
\end{cor}
\begin{proof}
Let $F(x) = Z(-\log x)$. To show the continuity of $F$, it suffices to show that $x_F= \infty$ or $F(x_F)= \infty$. 
If $F=Z_{w,v} \circ (-\log)$ for some $w \in E^0$ and if $v,w$ belong to the same component, then $x_F= \infty$ or 
$F(x_F) = \infty$ by Proposition \ref{prop:simplepole}, which implies $F \in \cF_1$. If $F=Z_{w,v} \circ (-\log)$ for some $w \in E^0$, then $F \in \cF_1$ by Proposition \ref{prop:partfuncformula} and the fact that $\cF_1$ is closed under the addition and the multiplication. 
Other generating functions can be decomposed into summations of products of generating functions of the form $Z_{w,v}$, which implies $F \in \cF_1$ in all cases. 
\end{proof}

\begin{dfn}
For a generating function $Z$ in Corollary \ref{cor:generatingfunctions}, the equivalence class of $Z \circ (-\log)$ is denoted by $[Z]$. 
\end{dfn}

\begin{lem} \label{lem:partitionfunc}
Let $C \in \pi(E)$ and let $w, w' \in C$. Suppose $\rho(A_C)>0$ and let $\rho=\rho(A_C)^{-1}$. Then, we have
$[Z_{w',w}] = [Z_{w}^C] = [\Phi_\rho]$. 
\end{lem}
\begin{proof}
Applying Proposition \ref{prop:simplepole} to the strongly connected graph $E_C$, we have $[Z_{w',w}]=[\Phi_\rho]$ by .  
Since $E^*_Cw = \bigsqcup_{w' \in C} w'E^*w$, we have $Z_w^C = \sum_{w' \in C} Z_{w',w}$. Hence, 
\[ [Z_w^C] = \sum_{w' \in C} [Z_{w',w}] = |C| [\Phi_\rho] = [\Phi_\rho], \]
which completes the proof. 
\end{proof}

For $C \in \pi(E)$, $\rho(A_C)=0$ if and only if $|C|=1$. For such component $C$, we formally consider $\rho=\rho(A_C)^{-1}=\infty$ and 
 $[Z_w^C]=[\Phi_{\rho}]=1$. 

\begin{dfn} \label{def:product} 
For $C \in \pi(E)$ and $w,w' \in C$, the class $[Z_{w',w}]$ does not depend on the choice of $w,w'$, which is denoted by $\cZ(C)$. 
Moreover, for $\mu = e_1\cdots e_n \in B^*C_0$, let 
\[ \cZ_{\mu} := \cZ(C_0) \prod_{i=1}^n \cZ(S(e_i)) \in \cF. \]
\end{dfn}

Note that $\cZ(C) = [\Phi_\rho] \in \cF$, where $\rho=\rho(A_C)^{-1}$ by Lemma \ref{lem:partitionfunc}. 

\begin{lem} \label{lem:summation}
Let $C \in B(E,C_0)^0$ and let $w \in C$. 
Then, we have  
\[ [Z_{w,v}] = \sum_{\mu \in CB^*C_0} \cZ_\mu \mbox{ and } [Z_v] = \sum_{\mu \in B^*C_0} \cZ_\mu. \]
\end{lem}
\begin{proof}
For $\tilde{\mu}=e_1e_2\cdots e_n \in C\tilde{B}^*C_0$, let $Z_{\tilde{\mu}}$ be the function in Proposition \ref{prop:partfuncformula},  
and let $[Z_{\tilde{\mu}}]$ denote the equivalence class of $Z_{\tilde{\mu}} \circ (-\log)$. 
By Propoisition \ref{prop:partfuncformula}, we have
\[ [ Z_{\tilde{\mu}}] = \cZ(C_0) \prod_{i=1}^n \cZ(S(e_i)) = \cZ_{\varepsilon(\tilde{\mu})}. \]
Here, Lemma \ref{lem:partitionfunc} and Lemma \ref{lem:PropertiesEquiv} \eqref{equiv:scalar} are applied. 
Hence, by Proposition \ref{prop:partfuncformula}, we have 
\[ [Z_{w,v}] = \sum_{\mu \in CB^*C_0} \sum_{\tilde{\mu} \in \varepsilon^{-1}(\mu)} [Z_{\tilde{\mu}}] 
= \sum_{\mu \in CB^*C_0} |\varepsilon^{-1}(\mu)| \cZ_{\mu} = \sum_{\mu \in CB^*C_0} \cZ_{\mu} . \]

In addition, 
\[ [Z_v] = \sum_{w \in E^0} [Z_{w,v}] = \sum_{C \in B^0} \sum_{w \in C} [Z_{w,v}] 
= \sum_{C \in B^0} |C| \sum_{\mu \in CB^*C_0} \cZ_{\mu} = \sum_{\mu \in B^*C_0} \cZ_{\mu}, \]
which completes the proof. 
\end{proof}

\begin{lem} \label{lem:simplify}
For $\mu \in B^*C_0$, we have
\[ \cZ_{\mu} = [\Phi_{x_v}]^{|\critical(\mu)|}. \]
\end{lem}
\begin{proof}
Let $\mu=e_1 \cdots e_n$, let $\rho_i = \rho(A_{S(e_i)})^{-1}$ for $i=1,\cdots,n$, and let $\rho_0 = \rho(A_{C_0})^{-1}$. 
Then, by Lemma \ref{lem:partitionfunc}, we have $\cZ(S(e_i))=[\Phi_{\rho_i}]$ for $i=1,\cdots,n$ and $\cZ(C_0)=[\Phi_{\rho_0}]$. 
In addition, $x_{\Phi_{\rho_i}} = \rho_i$ in Definition \ref{def:functions} is greater than or equal to $x_v$ by Equation \eqref{eqn:seneta}, and 
$x_{\Phi_{\rho_i}} = x_v$ if and only if $C \in \critical{\mu}$. Hence, 
\[ \cZ_{\mu} = \cZ(C_0) \prod_{i=1}^n \cZ(S(e_i)) = \prod_{i=0}^n [\Phi_{\rho_i}] = [\Phi_{x_v}]^{|\critical(\mu)|} \]
by Lemma \ref{lem:PropertiesEquiv} \eqref{equiv:multiplication1} 
\end{proof}

Now we are ready to prove Theorem \ref{thm:main}. 

\begin{proof}[Proof of Theorem \ref{thm:main}]
Fix $w \in E^0$. 
We first claim that the limit 
\begin{equation} \label{eqn:limit}
\lim_{\beta \to \beta_v+0} \varphi_{\beta,v} (p_w)
\end{equation}
exists for any $w \in E^0$. 
If $wE^*v = \emptyset$, then the claim holds and the limit is zero by Proposition \ref{prop:measure}. Suppose $wE^*v \neq \emptyset$. 
Then, by Proposition \ref{prop:measure}, we have 
\[ \varphi_{\beta,v}(p_w) = \frac{Z_{w,v}(\beta)}{Z_v(\beta)}. \]
By Lemma \ref{lem:summation} and Lemma \ref{lem:simplify} combined with Remark \ref{rmk:totalorder}, we have 
\[  
[Z_{w,v}] = \sum_{\mu \in CB^*C_0} \cZ_\mu 
= \sum_{\mu \in CB^*C_0} [\Phi_{x_v}]^{|\critical(\mu)|} = [\Phi_{x_v}]^{M_C}, \]
where $M_C = \max \{ |\critical(\mu)| \colon \mu \in CB^*C_0 \}$. Similarly, we have 
\[ [ Z_v] = \sum_{\mu \in B^*C_0} \cZ_\mu = [\Phi_{x_v}]^M, \]
where $M = \max \{ |\critical(\nu)| \colon \nu \in B^*C_0 \}$. 
In particular, $[Z_{w,v}], [Z_v] \in \cF_1$, which implies that  the limit in Equation \eqref{eqn:limit} exists by Remark \ref{rmk:totalorder}. 
In addition, we always have $[Z_{w,v}] \preceq [Z_v]$ since $M_C \leq M$. 

The limit in Equation \eqref{eqn:limit} is nonzero if and only if $\Phi_{x_v}^M \approx \Phi_{x_v}^{M_C}$, which is equivalent to $M=M_C$. 
In conclusion, we have
\begin{equation} \label{eqn:conclusion2}
\lim_{\beta \to \beta_v+0} \varphi_{\beta,v} (p_w) = \begin{cases} 0 & \mbox{if } M_C < M \\ \mbox{positive} & \mbox{if } M_C=M. \end{cases}
\end{equation}
Hence \eqref{main:existence} of Theorem \ref{thm:main} follows from Proposition \ref{prop:GeneralLimit} and Remark \ref{rmk:diverge} combined with Proposition \ref{prop:measure}. 
In addition, \eqref{main:convex} follows from Equation \eqref{eqn:conclusion2} combined with Corollary \ref{cor:coefficients}. 
\end{proof}

\section{Examples} \label{sec:example}
We visit two examples of the convex decomposition of $\varphi_v$ in Theorem \ref{thm:main}.  
We introduce the next definition just to simplify notations. 
For a natural number $n$, let
\[ F_n(\beta) := \sum_{k=0}^\infty n^ke^{-\beta k} = \frac{1}{1-ne^{-\beta}}. \]
Note that $F_n(-\log x)=\Phi_{n^{-1}}(x)$. The role of $F_n$ is almost the same as that of $\Phi_{\alpha}$ in Section \ref{sec:main}. 

\begin{exm} \label{example:intermediate}
We examine the effect of non-minimal critical components in the convex decomposition of $\varphi_v$.  
Consider the graph $E$ with $E^0=\{v_1,v_2, u_1,u_2,u_3,u_4, w_1,w_2,w_3,w_4\}$ depicted as below: 

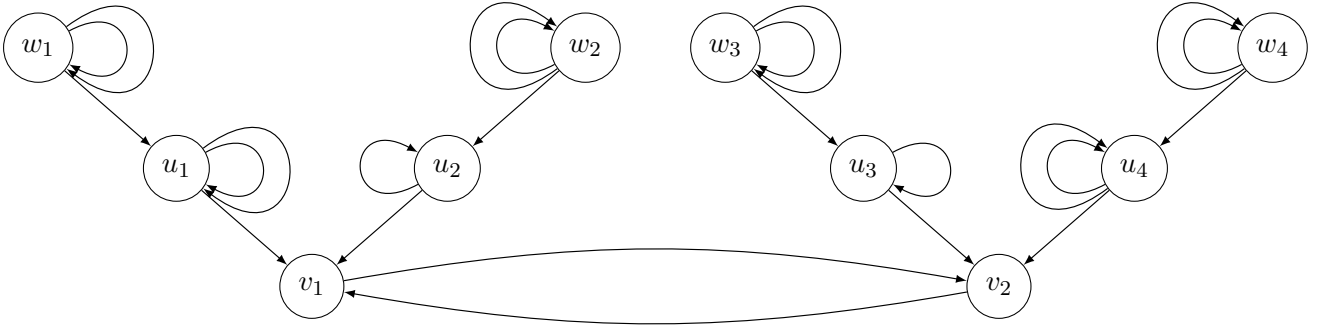
\begin{equation*}
\begin{tikzcd}[
arrow style=tikz,
>={latex} 
]
|[shape=circle,draw]| w_1 \arrow[rd]  \arrow[loop right, xscale=2, yscale=5] \arrow[loop right, xscale=3, yscale=10] & & & & 
|[shape=circle,draw]| w_2 \arrow[ld] \arrow[loop left, xscale=2, yscale=5] \arrow[loop left, xscale=3, yscale=10] & 
|[shape=circle,draw]| w_3 \arrow[rd]  \arrow[loop right, xscale=2, yscale=5] \arrow[loop right, xscale=3, yscale=10] & & & & 
|[shape=circle,draw]| w_4 \arrow[ld] \arrow[loop left, xscale=2, yscale=5] \arrow[loop left, xscale=3, yscale=10] \\
& |[shape=circle,draw]| u_1 \arrow[rd] \arrow[loop right, xscale=2, yscale=5] \arrow[loop right, xscale=3, yscale=10] & & 
|[shape=circle,draw]| u_2 \arrow[ld] \arrow[loop left, xscale=2, yscale=5] & & & 
|[shape=circle,draw]| u_3 \arrow[rd] \arrow[loop right, xscale=2, yscale=5] & & 
|[shape=circle,draw]| u_4 \arrow[ld] \arrow[loop left, xscale=2, yscale=5]  \arrow[loop left, xscale=3, yscale=10] & \\
& & |[shape=circle,draw]| v_1 \arrow[rrrrr, bend left=10]  & & & & &  |[shape=circle,draw]| v_2 \arrow[lllll, bend left=10]& & 
\end{tikzcd}
\end{equation*}

In this graph,  we can see that $\pmc(E)=\crit{v_j}(E) = \{C_1,C_2,C_3,C_4\}$, where $j=1,2$ and $C_i = \{w_i\}$ for $i=1,2,3,4$. Hence type III extremal KMS states at $\beta=\log 2$ are precisely $\psi_{C_i}$ for $i=1,2,3,4$ by Theorem \ref{thm:aHLRS15}. 
Let $C_0=\{v_1,v_2\} \in \pi(E)$, and let $D_i=\{u_i\}$ for $i=1,2,3,4$. For each $i=1,2,3,4$, there is a unique path $\mu_i$ in $B(E,C_0)$ going through $C_i$, $D_i$ and $C_0$. Then, $\mu_1$ and $\mu_4$ are maximal paths while $\mu_2$ and $\mu_3$ are not. 
Hence, the KMS states $\varphi_{v_j}$ for $j=1,2$ in Theorem \ref{thm:main} 
are convex combinations of $\psi_{C_1}$ and $\psi_{C_4}$, by the effect of the intermediate critical components $D_1$ and $D_4$. Here, by a concrete calculation, we verify this consequence and determine the coefficients in the convex combinations. 

First, we calculate $Z_{v_1}$ and $Z_{w_i,v_1}$ for $i=1,2,3,4$. 
For a concrete calculation, Proposition \ref{prop:partfuncformula} is applicable. We can naturally idenfity $B$ and $\tilde{B}$ in this example, because there are no multiple edges between components. We have 

\begin{align*}
Z_{w_1,v_1} &= Z_{\mu_1}(\beta) = F_2(\beta)e^{-\beta}F_2(\beta)e^{-\beta}\sum_{k=0}^\infty e^{-2k\beta} = e^{-2\beta}F_1(2\beta)F_2(\beta)^2, \\
Z_{w_2,v_1} &= Z_{\mu_2}(\beta) = F_2(\beta)e^{-\beta}F_1(\beta)e^{-\beta}\sum_{k=0}^\infty e^{-2k\beta} = e^{-2\beta}F_1(2\beta)F_1(\beta)F_2(\beta), \\
Z_{w_3,v_1} &= Z_{\mu_3}(\beta) = F_2(\beta)e^{-\beta}F_1(\beta)e^{-\beta}\sum_{k=0}^\infty e^{-(2k+1)\beta} = e^{-3\beta}F_1(2\beta)F_1(\beta)F_2(\beta), \\
Z_{w_4,v_1} &= Z_{\mu_4}(\beta) = F_2(\beta)e^{-\beta}F_2(\beta)e^{-\beta}\sum_{k=0}^\infty e^{-(2k+1)\beta} = e^{-3\beta}F_1(2\beta)F_2(\beta)^2. 
\end{align*}

Let $\nu_i$ $(i=1,2,3,4)$ be the unique paths in $B$ going through $D_i$ and $C_0$. Let $\xi=C_0$, which considered as a path of length $0$. 
Then, we have

\begin{align*}
Z_{v_1}(\beta) =\ \ \ & \sum_{i=1}^4 Z_{\mu_i} + \sum_{i=1}^4 Z_{\nu_i} +Z_{\xi} 
= \sum_{i=1}^4 \left( Z_{\mu_i} + Z_{\nu_i} \right) + Z_\xi \\
=\ \ \ & \sum_{k=0}^\infty e^{-2k\beta} \left( e^{-\beta}F_2(\beta) + e^{-\beta}F_2(\beta)e^{-\beta}F_2(\beta)\right) \\
+& \sum_{k=0}^\infty e^{-2k\beta} \left( e^{-\beta}F_1(\beta) + e^{-\beta}F_1(\beta)e^{-\beta}F_2(\beta)\right) \\
+& \sum_{k=0}^\infty e^{-(2k+1)\beta} \left( e^{-\beta}F_1(\beta) + e^{-\beta}F_1(\beta)e^{-\beta}F_2(\beta)\right) \\
+&\sum_{k=0}^\infty e^{-(2k+1)\beta} \left( e^{-\beta}F_2(\beta) + e^{-\beta}F_2(\beta)e^{-\beta}F_2(\beta)\right) +F_1(\beta)\\
=\ \ \ & F_1(\beta) \left( 1 + e^{-\beta}F_1(\beta) + e^{-\beta}F_2(\beta) + e^{-2\beta}F_1(\beta)F_2(\beta) + e^{-2\beta} F_2(\beta)^2\right). 
\end{align*}

Hence, we conclude that 
\[ 
\lim_{\beta \to \log 2+0} \frac{Z_{w_1,v_1}(\beta)}{Z_{v_1}(\beta)} = \frac23,\ 
\lim_{\beta \to \log 2+0} \frac{Z_{w_2,v_1}(\beta)}{Z_{v_1}(\beta)} = \lim_{\beta \to \log 2+0} \frac{Z_{w_3,v_1}}{Z_{v_1}} = 0,\ 
\lim_{\beta \to \log 2+0} \frac{Z_{w_1,v_1}(\beta)}{Z_{v_1}(\beta)} = \frac13, 
\]
which implies that 
\[ \varphi_{v_1} = \frac23 \psi_{C_1} + \frac13 \psi_{C_4}. \]
In addition we can see that 
\[ \varphi_{v_2} = \frac13 \psi_{C_1} + \frac23 \psi_{C_4} \]
in exactly the same way. 
\end{exm}

\begin{exm} \label{example:further}
We examine an effect of minimal components whose spectral radius is less than $e^{\beta_v}$ in the convex decomposition of $\varphi_v$. 
Consider the graph $E$ with $E^0=\{v, w_1,w_2, u_1,u_2\}$ depicted as below: 
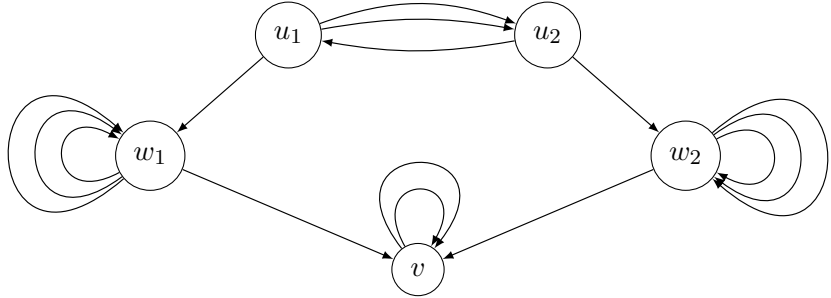
\begin{equation*}
\begin{tikzcd}[
arrow style=tikz,
>={latex} 
]
& |[shape=circle,draw]| u_1 \arrow[ld] \arrow[rr, bend left=10]  \arrow[rr, bend left=20]  & &  
|[shape=circle,draw]| u_2 \arrow[rd] \arrow[ll, bend left=10] & \\
|[shape=circle,draw]| w_1 \arrow[rrd]  \arrow[loop left, xscale=2, yscale=5] \arrow[loop left, xscale=3, yscale=10] \arrow[loop left, xscale=4, yscale=15] & & & & |[shape=circle,draw]| w_2 \arrow[lld] \arrow[loop right, xscale=2, yscale=5] \arrow[loop right, xscale=3, yscale=10]  \arrow[loop right, xscale=4, yscale=15] \\
& & |[shape=circle,draw]| v \arrow[loop above, xscale=5, yscale=2] \arrow[loop above, xscale=10, yscale=3]& & 
\end{tikzcd}
\end{equation*}
Let $C_0=\{v\}$, $C_i=\{w_i\}$ for $i=1,2$ and $C=\{u_1,u_2\}$. Then we have $\pmc(E)=\{C_1,C_2,C\}$ and $\crit{v}(E)=\{C_1,C_2\}$. Let $\varphi_v$ the KMS state in Theorem \ref{thm:main} and 
let $a_i=\varphi_v(p_{w_i})$ for $i=1,2$. Then, by Theorem \ref{thm:main} and Proposition \ref{prop:coefficients}, we have 
\begin{equation} \label{eqn:example}
\varphi = (\harm_{u_1}^{C_1} + \harm_{u_2}^{C_1}+a_1)\psi_{C_1} + (\harm_{u_1}^{C_2} + \harm_{u_2}^{C_2}+a_2)\psi_{C_2}, 
\end{equation} 
where 
\[ \begin{bmatrix} \harm_{u_1}^{C_i} \vspace{0.2cm} \\ \harm_{u_2}^{C_i} \end{bmatrix} = (3I-A_{C})^{-1} A_{C,C_i} 
[a_i]. \]
Hence, 
\[ \harm_{u_1}^{C_1} = \frac37a_1,\ \harm_{u_2}^{C_1} = \frac17a_1,\ \harm_{u_1}^{C_2} = \frac27a_2,\ \harm_{u_2}^{C_2} = \frac37a_2. \]
In addition, we can directly see that $Z_{w_1,v}=Z_{w_2,v}$, which implies that $a_1=a_2$. Since the sum of coefficients in Equation \eqref{eqn:example} is equal to $1$, we have $a=7/23$. Consequently, we have
\[ \varphi_v = \frac{11}{23}\psi_{C_1} + \frac{12}{23}\psi_{C_2}. \]
By the effect of the minimal component $C$, coefficients of $\psi_{C_i}$ are not the same. 
\end{exm}

\appendix
\section{A direct proof of Theorem \ref{thm:aHLRS15}} \label{sec:appendix}

We include a direct proof of Theorem \ref{thm:aHLRS15} for reader's convenience. The essence of the proof is identical to the original one. Note that for a finite graph $E$, we have 
\begin{equation} \label{eqn:seneta}
 \rho(A_E) = \max \{ \rho(A_C) \colon C \in \pi(E)\} 
\end{equation}
by \cite[Section 4]{aHLRS15}. 

\begin{lem} \label{lem:subinvariance}
Let $\beta >0$ and suppose that there exists a $\beta$-harmonic vector $\harm$ of $E$. 
Then $\beta = \beta_C$ for some $C \in \pmc(E)$. 
Moreover, for each $C \in \pi(E)$, we have 
$\harm_C = 0$ if either $\rho(A_C)>e^\beta$, or $\rho(A_C)=e^\beta$ and $C \not \in \pmc(E)$. 
\end{lem}

\begin{proof}
we have 
\begin{equation} \label{eqn:subinv} 
e^\beta \harm_C = (A_E \harm)_C = \sum_{D \in \pi(E)} A_{C,D} \harm_D \geq A_C \harm_C 
\end{equation}
for all $C \in \pi(E)$. Then, the subinvariance theorem (see \cite[Theorem 1.6]{Sen}) implies that if $e^\beta < \rho(A_C)$, then $\harm_C = 0$. Hence, $\harm_L$ is an eigenvector of $A_L$ with eigenvalue $e^\beta$, where 
\[ L = \bigcup \{ C \in \pi(E) \colon \rho(A_C) \leq e^\beta \}. \]
If $\beta \neq \log \rho(A_{C})$ for any $C \in \pi(E)$, then we have $\rho(A_L) = \max \{ \rho(A_C) \colon \rho(A_C) < e^\beta \} <e^\beta$ by Equation \eqref{eqn:seneta}, which contradicts the fact that $e^\beta$ is an eigenvalue of $A_L$. Hence, there exists $C \in \pi(E)$ such that $\beta = \log \rho(A_{C})$. 

Let $C \in \pi(E)$ be a component satisfying $\harm_C \neq 0$ and $\rho(A_C) = e^\beta$ (which implies that $\harm_v>0$ for any $v \in C$). It suffices to show $C \in \pmc(E)$. Supoose $D \in \pi(E) \setminus \{C\}$ is a component such that $DE^*C \neq \emptyset$. Then, by the same argument as in Equation \eqref{eqn:subinv}, we have 
\begin{equation} \label{eqn:subinvariance} 
e^\beta \harm_D \geq A_D \harm_D + A_{D,C} \harm_C \geq A_D \harm_D. 
\end{equation}
Since $A_{D,C} \neq 0$ and $\harm_w >0$ for all $w \in C$, we have $A_{D,C} \harm_C \neq 0$, which implies that $\harm_D \neq 0$. The subinvariance theorem implies that either $\rho(A_D) < e^\beta$, or $\rho(A_D) = e^\beta$ and the equality in Inequation \eqref{eqn:subinvariance} holds. The latter case implies that $A_{D,C}\harm_C =0$, which is a contradiction. Hence we have $\rho(A_D) < e^\beta$, which implies $C \in \pmc(E)$.  
\end{proof}

\begin{proof}[Proof of Theorem \ref{thm:aHLRS15}]
Let $\harm$ be a $\beta$-harmonic vector of $E$. 
We show that $\harm$ is completely determined from $\harm_K$. Let 
\[ L = \bigcup \{ C \in \pi(E) \colon \rho(A_C) < e^\beta \},\ 
\tilde{S} = K \cup L,\ 
S = S_K.   \]
Note that $K \subset S \subset \tilde{S}$ by the minimality of all components in $K$. 
By Lemma \ref{lem:subinvariance}, $\harm_C=0$ if $C \in \pi(E) \setminus \pi(E_{\tilde{S}})$. Hence, $\harm_{\tilde{S}}$ is an eigenvector of $A_{\tilde{S}}$ with eigenvalue $e^\beta$. Since $A_{\tilde{S}}$ is of the form 
\[ A_{\tilde{S}} = \begin{bmatrix} A_{S} & * \\ 0 & A_{\tilde{S}\setminus S} \end{bmatrix}, \]
the vector $\harm_{\tilde{S}\setminus S}$ is an eigenvector of $A_{\tilde{S}\setminus S}$ with eigenvalue $e^\beta$ if it is nonzero. 
However, since $\tilde{S}\setminus S \subseteq L$, we have $\rho(A_{\tilde{S}\setminus S}) < e^\beta$ by Equation \eqref{eqn:seneta},  
which implies that $\harm_{\tilde{S}\setminus S}=0$. Consequently, we have $\harm_{E^0\setminus S}=0$ and $\harm_S$ is an eigenvector of $A_S$ with eigenvalue $e^\beta$. 

Let $B = A_{S\setminus K, K}$. Since $A_S$ is of the form
\[ A_S = \begin{bmatrix} A_{S\setminus K} & B \\ 0 & A_K \end{bmatrix}, \]
we have
\begin{equation} \label{eqn:main}
0 = \begin{bmatrix} e^\beta I-A_{S\setminus K} & -B \\ 0 & e^\beta I-A_K \end{bmatrix} 
\begin{bmatrix} \harm_{S\setminus K} \\ \harm_K \end{bmatrix}. 
\end{equation}
We have $\rho(A_{S\setminus K}) < e^\beta$ by Equation \eqref{eqn:seneta} and the fact that $S\setminus K \subset L$. Hence we multiply the matrix 
\[ \begin{bmatrix} (e^\beta I-A_{S\setminus K})^{-1} & 0 \\ 0 & I \end{bmatrix} \]
to the left and right hand sides of Equation \eqref{eqn:main}, so that we have 
\[
0 = \begin{bmatrix} I & -(e^\beta I-A_{S\setminus K})^{-1}B \\ 0 & e^\beta I-A_K \end{bmatrix} 
\begin{bmatrix} \harm_{S\setminus K} \\ \harm_K \end{bmatrix}, 
\]
and hence 
\[ \harm_{S\setminus K} -  (e^\beta I-A_{S\setminus K})^{-1}B \harm_K = 0,\ (e^\beta I-A_K)\harm_K=0. \]
Therefore, $\harm_K$ is a $\beta$-harmonic vector of $E_K$, which implies \eqref{item:restriction} in Theorem \ref{thm:aHLRS15}. 
In addition, $\harm_{S\setminus K}$ is determined from $\harm_K$, which implies the uniqueness of \eqref{item:uniqueness}. 

Conversely, a $\beta$-harmonic vector $\harm_K$ of $E_K$ extends to a $\beta$-harmonic vector $\harm_S$ of $E_S$ by letting 
\[ \harm_{S\setminus K} = (e^\beta I-A_{S\setminus K})^{-1}B \harm_K \] 
from the same argument as above. Since $A_E$ is of the form 
\[ A_E = \begin{bmatrix} A_{S} & * \\ 0 & A_{E^0\setminus S} \end{bmatrix}, \]
the vector $\harm_S$ exntends to a $\beta$-harmonic vector $\harm$ of $E$ by letting $\harm_{E^0\setminus S} = 0$. Hence, the existence in \eqref{item:uniqueness} and \eqref{item:extension} hold. 
\end{proof}

\bibliographystyle{amsplain}
\bibliography{Graph}

\begin{bibdiv}
\begin{biblist}

\bib{aHLRS13}{article}{
      author={an~Huef, A.},
      author={Laca, M.},
      author={Raeburn, I.},
      author={Sims, A.},
       title={K{MS} states on the {$C^*$}-algebras of finite graphs},
        date={2013},
     journal={J. Math. Anal. Appl.},
      volume={405},
      number={2},
       pages={388\ndash 399},
}

\bib{aHLRS15}{article}{
      author={an~Huef, A.},
      author={Laca, M.},
      author={Raeburn, I.},
      author={Sims, A.},
       title={K{MS} states on the {$C^\ast$}-algebras of reducible graphs},
        date={2015},
     journal={Ergodic Theory Dynam. Systems},
      volume={35},
      number={8},
       pages={2535\ndash 2558},
}

\bib{BR2}{book}{
      author={Bratteli, O.},
      author={Robinson, D.~W.},
       title={Operator algebras and quantum statistical mechanics. 2},
     edition={Second},
      series={Texts and Monographs in Physics},
   publisher={Springer-Verlag, Berlin},
        date={1997},
        note={Equilibrium states. Models in quantum statistical mechanics},
}

\bib{BT1}{article}{
      author={Bruce, C.},
      author={Takeishi, T.},
       title={{C}*-dynamical invariants and {T}oeplitz algebras of graphs},
        date={2021},
     journal={arXiv:2103.02816},
}

\bib{CarLar}{article}{
      author={Carlsen, T.~M.},
      author={Larsen, N.~S.},
       title={Partial actions and {KMS} states on relative graph
  {$C^*$}-algebras},
        date={2016},
     journal={J. Funct. Anal.},
      volume={271},
      number={8},
       pages={2090\ndash 2132},
}

\bib{EFW}{article}{
      author={Enomoto, M.},
      author={Fujii, M.},
      author={Watatani, Y.},
       title={K{MS} states for gauge action on {$O_{A}$}},
        date={1984},
     journal={Math. Japon.},
      volume={29},
      number={4},
       pages={607\ndash 619},
}

\bib{FowRae}{article}{
      author={Fowler, N.~J.},
      author={Raeburn, I.},
       title={The {T}oeplitz algebra of a {H}ilbert bimodule},
        date={1999},
     journal={Indiana Univ. Math. J.},
      volume={48},
      number={1},
       pages={155\ndash 181},
}

\bib{OlePed}{article}{
      author={Olesen, D.},
      author={Pedersen, G.~K.},
       title={Some {$C^{\ast} $}-dynamical systems with a single {KMS} state},
        date={1978},
     journal={Math. Scand.},
      volume={42},
      number={1},
       pages={111\ndash 118},
}

\bib{Pat}{article}{
      author={Paterson, A. L.~T.},
       title={Graph inverse semigroups, groupoids and their
  {$C^\ast$}-algebras},
        date={2002},
     journal={J. Operator Theory},
      volume={48},
      number={3, suppl.},
       pages={645\ndash 662},
}

\bib{Sen}{book}{
      author={Seneta, E.},
       title={Non-negative matrices and {M}arkov chains},
      series={Springer Series in Statistics},
   publisher={Springer, New York},
        date={2006},
        note={Revised reprint of the second (1981) edition.},
}

\bib{Th17}{article}{
      author={Thomsen, K.},
       title={K{MS} weights on graph {$C^*$}-algebras},
        date={2017},
     journal={Adv. Math.},
      volume={309},
       pages={334\ndash 391},
}

\bib{Web}{article}{
      author={Webster, S. B.~G.},
       title={The path space of a directed graph},
        date={2014},
     journal={Proc. Amer. Math. Soc.},
      volume={142},
      number={1},
       pages={213\ndash 225},
}

\end{biblist}
\end{bibdiv}

\end{document}